\documentclass[12pt,twoside,reqno]{amsart}
\usepackage{amsmath}
\usepackage{amsfonts}
\usepackage{amssymb}
\usepackage{xcolor}
\usepackage{graphicx}
\usepackage{mathtools}
\usepackage{mathrsfs}
\usepackage{cite}
\usepackage{cleveref}
\usepackage{geometry}
\usepackage{marginnote}
\usepackage{todonotes}
\newtheorem{theorem}{Theorem}[section]
\newtheorem{corollary}[theorem]{Corollary}
\newtheorem{lemma}[theorem]{Lemma}
\newtheorem{proposition}[theorem]{Proposition}

\newtheorem{remark}[theorem]{Remark}
\allowdisplaybreaks
\numberwithin{equation}{section}
\begin{document}
\title{A sharpened Carlson's integral inequality} 
\thanks{\texttt{orcid: 0000-0003-3091-8085}}

\author{Philippe Lauren\c{c}ot}
\address{Laboratoire de Math\'ematiques (LAMA) UMR~5127, Universit\'e Savoie Mont Blanc, CNRS, F--73000 Chamb\'ery, France}
\email{philippe.laurencot@univ-smb.fr}


\date{\today}

\begin{abstract}
A sharpened version of Carlson's integral inequality is established, which features a remainder term involving a relative distance to the set of extremal functions. A similar approach leads to an alternative derivation of Carlson's and Landau's discrete inequalities and provides another characterization of the optimal constant in Carlson's inequality for finite sums.  
\end{abstract}

\maketitle

%
%
\pagestyle{myheadings}
\markboth{\sc{Ph. Lauren\c{c}ot}}{\sc{Carlson's inequalities}}

\section{Introduction}\label{sec1}

In \cite{Car1934}, Carlson proves that, for any non-zero sequence $\mathbf{a}=(a_n)_{n\ge 1}$ of non-negative real numbers,
\begin{equation}
	\left( \sum_{n=1}^\infty a_n \right)^4 < \pi^2 \left( \sum_{n=1}^\infty a_n^2 \right) \left( \sum_{n=1}^\infty n^2 a_n^2 \right), \label{DCI} 
\end{equation}
the constant $\pi^2$ being the best possible, see also \cite{Har1936} for simpler proofs, as well as \cite[Ch.~1]{LMPP2006} and \cite[Ch.~VIII \S~1]{MPF1991} for more information about Carlson's inequality and its extensions. An integral counterpart of~\eqref{DCI} is also derived in \cite{Car1934}: it states that every non-negative function $f\in L^2((0,\infty),(1+x^2) \mathrm{d}x)$ belongs to $ L^1((0,\infty))$ and satisfies
\begin{equation}
	\left( \int_0^\infty f(x)\ \mathrm{d}x \right)^4 \le \pi^2 \left( \int_0^\infty f^2(x)\ \mathrm{d}x \right) \left( \int_0^\infty x^2 f^2(x)\ \mathrm{d}x \right), \label{CI}
\end{equation}
the constant $\pi^2$ being again the best possible, see also \cite[Ch.~5 \S~8]{BeBe1983}. More precisely, equality holds in~\eqref{CI} if and only if there are $(A,\lambda,\mu)\in [0,\infty)\times (0,\infty)^2$ such that
\begin{equation*}
	f(x) = \frac{A}{\mu + \lambda x^2}, \qquad x\in (0,\infty),
\end{equation*}
a feature which contrasts markedly with Carlson's discrete inequality~\eqref{DCI} for which there is no extremal sequence due to the strict inequality in~\eqref{DCI}. Replacing $f$ by $f/\|f\|_{L^1((0,\infty))}$ in~\eqref{CI}, we observe that~\eqref{CI} is equivalent to the minimizing problem
\begin{equation}
	\frac{1}{\pi} = \min_{f\in\mathcal{C}_1} \mathcal{J}[f] \;\;\text{ with }\;\; \mathcal{J}[f] := \left( \int_0^\infty f^2(x)\ \mathrm{d}x \right)^{1/2} \left( \int_0^\infty x^2 f^2(x)\ \mathrm{d}x \right)^{1/2}, \label{CIE}
\end{equation}
where
\begin{equation*}
	\mathcal{C}_1 := \left\{ f\in L^2((0,\infty),(1+x^2) \mathrm{d}x)\ : \ f\ge 0 \;\text{ a.e. and }\; \|f\|_{L^1((0,\infty))}=1 \right\}.
\end{equation*} 
In addition, $f\in\mathcal{C}_1$ is a minimizer of~\eqref{CIE} (i.e. $\mathcal{J}[f]=1/\pi$) if and only if $f\in\mathcal{E} := \{\varphi_z\ :\ z>0\}$, where
\begin{equation}
	\varphi_z(x) := \frac{2}{\pi} \frac{z}{1+z^2 x^2}, \qquad (x,z)\in (0,\infty)^2. \label{CIM}
\end{equation}

We aim at showing the following refined version of~\eqref{CIE}.

\begin{theorem}\label{thm1}
	Let $f\in \mathcal{C}_1$. Then 
	\begin{equation}
		\mathcal{J}[f] = \left( \int_0^\infty f^2(x)\ \mathrm{d}x \right)^{1/2} \left( \int_0^\infty x^2 f^2(x)\ \mathrm{d}x \right)^{1/2} \ge \frac{1}{\pi} + \frac{1}{\pi} d^2(f,\mathcal{E}),  \label{CIR01}
	\end{equation}	
	where
	\begin{equation}
		d^2(f,\mathcal{E}) := \inf_{z>0} \int_0^\infty \frac{(f-\varphi_{z})^2(x)}{\varphi_{z}(x)}\ \mathrm{d}x. \label{CIR02}
	\end{equation}
	Moreover, the constant $1/\pi$ is optimal. In fact, introducing
	\begin{equation}
		\varrho(f) := \frac{1}{d^2(f,\mathcal{E})} \left[ \mathcal{J}[f] - \frac{1}{\pi} \right] \label{CIR03}
	\end{equation}
	for $f\in \mathcal{C}_1$, we have
	\begin{equation}
		\inf_{f\in\mathcal{C}_1} \varrho(f) = \min_{f\in\mathcal{C}_1} \varrho(f) = \frac{1}{\pi} = \varrho(\psi_z), \qquad z\in (0,\infty), \label{CIR04}
	\end{equation}
	with
	\begin{equation}
		\psi_z(x) := \frac{4}{\pi} \frac{z}{(1+z^2 x^2)^2}, \qquad (x,z)\in (0,\infty)^2. \label{CIR05}
	\end{equation}
\end{theorem}

The main building block of the proof is the following identity, which seems to be unnoticed so far and provides an alternative formula for the left-hand side of~\eqref{CIR01}.

\begin{proposition}\label{prop1}
	Let $f\in \mathcal{C}_1$. Then 
	\begin{equation}
		\mathcal{J}[f] = \frac{1}{\pi} + \frac{1}{\pi} \int_0^\infty \frac{(f-\varphi_{r_f})^2(x)}{\varphi_{r_f}(x)}\ \mathrm{d}x \label{CIR06}
	\end{equation}
	with
	\begin{equation}
		r_f^2 := \left( \int_0^\infty f^2(x)\ \mathrm{d}x \right) \left( \int_0^\infty x^2 f^2(x)\ \mathrm{d}x \right)^{-1}, \label{CIR07}
	\end{equation}
	the functional $\mathcal{J}$ and the function $\varphi_z$, $z>0$, being defined in~\eqref{CIE} and~\eqref{CIM}, respectively.
\end{proposition}

We now show that the technique used to prove \Cref{thm1} applies as well to  Carlson's discrete inequality~\eqref{DCI}. Let us first recall that it is actually shown in \cite{Car1934} that
\begin{equation*}
	\left( \sum_{n=1}^\infty a_n \right)^4 < \pi^2 \left( \sum_{n=1}^\infty a_n^2 \right) \left( \sum_{n=1}^\infty \left[ \left( n - \frac{1}{2}\right)^2 + \frac{3}{16} \right] a_n^2 \right) 
\end{equation*}
for any non-zero sequence $\mathbf{a}=(a_n)_{n\ge 1}$ of non-negative real numbers, the constant $\pi^2$ being the best possible. The above inequality obviously implies~\eqref{DCI} and is subsequently improved by Landau to the following inequality
\begin{equation}
	\left( \sum_{n=1}^\infty a_n \right)^4 < \pi^2 \left( \sum_{n=1}^\infty a_n^2 \right) \left( \sum_{n=1}^\infty \left( n - \frac{1}{2}\right)^2  a_n^2 \right) \label{DLI}
\end{equation}
for any non-zero sequence $\mathbf{a}=(a_n)_{n\ge 1}$ of non-negative real numbers, the constant $\pi^2$ being again the best possible, see \cite[Eq.~(2.7)]{LMPP2006}. As in the continuous case, the inequalities~\eqref{DCI} and~\eqref{DLI} are equivalent to the minimizing problem
\begin{equation}
	\frac{1}{\pi} = \inf_{\mathbf{a}\in\mathcal{D}_1} \left\{ \left( \sum_{n=1}^\infty a_n^2 \right)^{1/2} \left( \sum_{n=1}^\infty \left( n - \theta \right)^2 a_n^2 \right)^{1/2} \right\}, \qquad \theta\in \left\{0,\frac{1}{2}\right\}, \label{DLIE} 
\end{equation}
where
\begin{equation*}
	\mathcal{D}_1 := \left\{ \mathbf{a}=(a_n)_{n\ge 1}\ :\ \sum_{n=1}^\infty n^2 a_n^2 < \infty, \ \sum_{n=1}^\infty a_n =1 \;\text{ and }\; a_n\ge 0, \ n\ge 1 \right\}.
\end{equation*}
We are then led to study
\begin{equation}
	\mu(\theta) := \inf_{\mathbf{a}\in\mathcal{D}_1} \big\{ J_\theta[\mathbf{a}] \big\} \;\;\text{ with }\;\; J_\theta[\mathbf{a}] :=  \left( \sum_{n=1}^\infty a_n^2 \right)^{1/2} \left( \sum_{n=1}^\infty \left( n - \theta \right)^2 a_n^2 \right)^{1/2} \label{DCIR01}
\end{equation}
for $\theta\in [0,1)$. According to~\eqref{DLIE} and the monotonicity of $\theta\mapsto\mu(\theta)$, 
\begin{equation}
	\mu(0) = \mu(1/2) = \mu(\theta) = \frac{1}{\pi}, \qquad \theta\in [0,1/2]. \label{DCIR02}
\end{equation}
We shall see now that the range of $\theta$ for which~\eqref{DCIR02} is valid is optimal.

\begin{theorem}\label{thm2}
Let $\theta\in [0,1)$ and define
\begin{equation}
	\nu_\theta(z) := \sum_{n=1}^\infty \frac{2z}{1+z^2(n-\theta)^2}, \qquad z\in (0,\infty). \label{DCIR03}
\end{equation}
Introducing
\begin{equation}
	\bar{\nu}_\theta := \sup_{z\in (0,1/(1-\theta)]} \nu_\theta(z) = \sup_{z\in (0,1/(1-\theta))} \nu_\theta(z), \label{DCIR04}
\end{equation}
there holds $\mu(\theta) = 1/\bar{\nu}_\theta$. In addition,
\begin{equation}
	\bar{\nu}_\theta = \pi \;\;\text{ if and only if}\;\; \theta\in [0,1/2]. \label{DCIR05}
\end{equation}
More precisely, 
\begin{itemize}
	\item [$\triangleright$] for $\theta\in [0,1/2]$, there holds $J_\theta[\mathbf{a}]>\mu(\theta)=1/\pi$ for all $\mathbf{a}\in\mathcal{D}_1$;
	\item [$\triangleright$] for $\theta\in (1/2,1)$, $\mu(\theta)<1/\pi$ and the set $\mathcal{S}_\theta := \{z\in (0,1/(1-\theta))\ :\ \nu_\theta(z)=\bar{\nu}_\theta\}$ is not empty. Furthermore, $J_\theta[\mathbf{a}]= \mu(\theta)$ if and only if there is $Z\in \mathcal{S}_\theta$ such that $\mathbf{a}=\boldsymbol{\varphi}_Z$, where $\boldsymbol{\varphi}_Z = \big(\varphi_{n,z}\big)_{n\ge 1}$ is defined in~\eqref{DCIR08} below. 
\end{itemize} 
\end{theorem}

There is thus a sharp dichotomy with respect to $\theta\in [0,1)$ in the minimizing problem~\eqref{DCIR01}. Either the infimum $\mu(\theta)$ equals $1/\pi$ and is not attained ($\theta\in [0,1/2]$). Or the infimum $\mu(\theta)$ is strictly less than $1/\pi$ and there are minimizers $(\theta\in (1/2,1))$.

\Cref{thm2} (with $\theta=0$) applies of course to sequences which have only a finite number $N\ge 2$ of non-zero entries. However, $1/\pi$ is no longer the best constant in that case and the following improved bound is derived in \cite[Theorem~1 \& Remark~1]{LPP2005}, see also \cite[Section~1.4]{LMPP2006}. For $N\ge 2$, 
\begin{equation}
	 \frac{1}{\pi} < \frac{1}{2\arctan{N}} < \mu_N := \inf_{\mathbf{a}\in\mathcal{D}_{1,N}} S_N[\mathbf{a}], \label{FSCI} 
\end{equation}
where 
\begin{equation*}
	S_N[\mathbf{a}] :=  \left( \sum_{n=1}^N a_n^2 \right)^{1/2} \left( \sum_{n=1}^N n^2 a_n^2 \right)^{1/2}, \qquad \mathbf{a}\in \mathcal{D}_{1,N}, 
\end{equation*}
and
\begin{equation*}
	\mathcal{D}_{1,N} := \left\{ \mathbf{a}=(a_n)_{1\le n\le N}\in [0,\infty)^N\ :\  \sum_{n=1}^N a_n =1 \right\}.
\end{equation*} 
As pointed out in \cite[Remark~4]{LPP2005}, $\mu_N > 1/(2\arctan{N})$ and the aim of the next result, which is proved along the lines of \Cref{thm2}, is to provide an alternative way to compute $\mu_N$ and the corresponding minimizers. Unfortunately, we do not obtain an explicit formula. but nevertheless improve the lower bound~\eqref{FSCI}.

\begin{theorem}\label{thm3}
Let $N\ge 2$. For $z>0$, we define $\boldsymbol{\varphi}_z=(\varphi_{n,z})_{1\le n \le N}\in \mathcal{D}_{1,N}$ by 
\begin{equation}
	\nu_N(z) := \sum_{n=1}^N \frac{2z}{1+n^2 z^2}, \qquad \varphi_{n,z} := \frac{1}{\nu_N(z)} \frac{2z}{1 + n^2 z^2},  \quad 1\le n \le N. \label{FSCI3}
\end{equation} 
Then
\begin{equation*}
	\frac{1}{\mu_N} = \bar{\nu}_N := \max_{[1/N,1]} \nu_N = \max_{(1/N,1)} \nu_N
\end{equation*}
and $\mathbf{a}\in\mathcal{D}_{1,N}$ satisfies $S_N[\mathbf{a}]=\mu_N$ if and only if there is $Z\in (1/N,1)$ such that $\nu_N(Z) = \bar{\nu}_N$ and $\mathbf{a}=\boldsymbol{\varphi}_{Z}$.
Moreover, setting
\begin{equation}
	Z_2^2 := \frac{1}{60} \left( 1 + (11251 - 90 \sqrt{12261})^{1/3} + (11251 + 90 \sqrt{12261}^{1/3}\right) \qquad (Z_2\sim 0.811), \label{Z2}
\end{equation}
there holds
\begin{equation}
	\mu_N \ge \frac{1}{2 \arctan(NZ_2)} > \frac{1}{2 \arctan N}. \label{FSCI4}
\end{equation}
\end{theorem}

As in \cite[Theorem~1 \& Remark~1]{LPP2005}, the proof of \Cref{thm3} ensures that
\begin{equation*}
	\lim_{N\to\infty} \mu_N = \pi,
\end{equation*}
see \Cref{prop3} below.

\section{Carlson's integral inequality with remainder}\label{sec2}

We first observe that, for $z>0$, the function $\varphi_z$ belongs to $\mathcal{C}_1$ with
\begin{equation}
	\int_0^\infty \varphi_z(x)\ \mathrm{d}x = 1, \quad \int_0^\infty \varphi_z^2(x)\ \mathrm{d}x = \frac{z}{\pi}, \quad \int_0^\infty x^2 \varphi_z^2(x)\ \mathrm{d}x = \frac{1}{\pi z}, \quad r_{\varphi_z} = z, \label{CIN}
\end{equation}
see \Cref{lemap1}.

\begin{proof}[Proof of \Cref{prop1}]
	Let $f\in \mathcal{C}_1$. For $\varphi\in \mathcal{C}_1$ to be determined later, we compute
	\begin{align*}
		q_{2,f} := \int_0^\infty x^2 f^2(x)\ \mathrm{d}x & = \int_0^\infty x^2 \varphi^2(x)\ \mathrm{d}x + 2 \int_0^\infty x^2 \varphi(x) (f-\varphi)(x)\ \mathrm{d}x \\
		& \qquad + \int_0^\infty x^2 (f-\varphi)^2(x)\ \mathrm{d}x
	\end{align*}
	and
	\begin{align*}
		q_{0,f} := \int_0^\infty f^2(x)\ \mathrm{d}x = \int_0^\infty \varphi^2(x)\ \mathrm{d}x + 2 \int_0^\infty \varphi(x) (f-\varphi)(x)\ \mathrm{d}x + \int_0^\infty (f-\varphi)^2(x)\ \mathrm{d}x.
	\end{align*}
	Combining the above two identities gives, for $z>0$,
	\begin{align*}
		\frac{z q_{2,f}}{2} + \frac{q_{0,f}}{2z} & = \int_0^\infty \left( \frac{z x^2}{2} + \frac{1}{2z} \right) \varphi^2(x)\ \mathrm{d}x + 2 \int_0^\infty \left( \frac{z x^2}{2} + \frac{1}{2z} \right) \varphi(x) (f-\varphi)(x)\ \mathrm{d}x \\
		& \quad + \int_0^\infty \left( \frac{z x^2}{2} + \frac{1}{2z} \right) (f-\varphi)^2(x)\ \mathrm{d}x.
	\end{align*}
	Since $\varphi_z\in\mathcal{C}_1$, it is rather natural to choose $\varphi=\varphi_z$, which cancels the second term on the right-hand side of the above inequality. With this choice, we deduce from~\eqref{CIM} and~\eqref{CIN} that
	\begin{align*}
		\frac{z q_{2,f}}{2} + \frac{q_{0,f}}{2z} & = \int_0^\infty \left( \frac{z x^2}{2} + \frac{1}{2z} \right) \varphi_z^2(x)\ \mathrm{d}x + \int_0^\infty \left( \frac{z x^2}{2} + \frac{1}{2z} \right) (f-\varphi_z)^2(x)\ \mathrm{d}x\\
		& = \frac{1}{\pi} \int_0^\infty \varphi_z(x)\ \mathrm{d}x +  \frac{1}{\pi} \int_0^\infty \frac{(f-\varphi_z)^2(x)}{\varphi_z(x)}\ \mathrm{d}x \\
		& = \frac{1}{\pi} +  \frac{1}{\pi} \int_0^\infty \frac{(f-\varphi_z)^2(x)}{\varphi_z(x)}\ \mathrm{d}x.
	\end{align*}
	We finally choose $z=r_f = \sqrt{q_{0,f}/q_{2,f}}$ in the above inequality to obtain~\eqref{CIR06}.
\end{proof}

\begin{proof}[Proof of \Cref{thm1}]
	Let $f\in\mathcal{C}_1$ with $r_f$ defined in~\eqref{CIR07}. Since $\varphi_{r_f}\in\mathcal{E}$, we have
	\begin{equation*}
		\int_0^\infty \frac{(f-\varphi_{r_f})^2(x)}{\varphi_{r_f}(x)}\ \mathrm{d}x \ge d^2(f,\mathcal{E}),
	\end{equation*}
	from which~\eqref{CIR01} readily follows by \Cref{prop1}. A direct consequence of~\eqref{CIR01} is then
	\begin{equation}
		\varrho(f) \ge \frac{1}{\pi}. \label{CIR08}
	\end{equation}
	To complete the proof, we are left with computing $\varrho(\psi_z)$ for $z>0$. Let $z>0$ and first note that \Cref{lemap1} guarantees that $\psi_z$ belongs to $\mathcal{C}_1$ with
	\begin{equation*}
		\int_0^\infty \psi_z^2(x)\ \mathrm{d}x = \frac{5z}{2\pi}, \quad \int_0^\infty x^2 \psi_z^2(x)\ \mathrm{d}x = \frac{1}{2\pi z}, \quad r_{\psi_z} = z\sqrt{5}.
	\end{equation*}
	We then infer from \Cref{prop1} that
	\begin{equation}
		\varrho(\psi_z) = \frac{1}{\pi d^2(\psi_z,\mathcal{E})} \int_0^\infty \frac{(\psi_z-\varphi_{z\sqrt{5}})^2(x)}{\varphi_{z\sqrt{5}}(x)}\ \mathrm{d}x. \label{CIR09} 
	\end{equation}
	Now, for $\zeta\in (0,\infty)$,
	\begin{align*}
		h(\zeta) & := \int_0^\infty \frac{(\psi_z-\varphi_{\zeta})^2(x)}{\varphi_{\zeta}(x)}\ \mathrm{d}x = \int_0^\infty \frac{\psi_z^2(x)}{\varphi_{\zeta}(x)}\ \mathrm{d}x - 2 \int_0^\infty \psi_z(x)\ \mathrm{d}z + \int_0^\infty \varphi_\zeta(x)\ \mathrm{d}x \\
		& = \frac{\pi}{2 \zeta} \int_0^\infty \big(1+\zeta^2 x^2\big) \psi_z^2(z)\ \mathrm{d}x - 1 = \frac{\pi}{2 \zeta} \left( \frac{5z}{2\pi} + \frac{\zeta^2}{2\pi z} \right) - 1 \\
		& = \frac{5z}{4\zeta} + \frac{\zeta}{4z} - 1 \ge \frac{\sqrt{5}}{2} = h(z\sqrt{5}).
	\end{align*}
	Consequently,
	\begin{equation*}
		d^2(\psi_z,\mathcal{E}) = \int_0^\infty \frac{(\psi_z-\varphi_{z\sqrt{5}})^2(x)}{\varphi_{z\sqrt{5}}(x)}\ \mathrm{d}x,
	\end{equation*}
	which gives, together with~\eqref{CIR09}, $\varrho(\psi_z)=1/\pi$. Combining this property with~\eqref{CIR08} completes the proof.
\end{proof}

\begin{remark}\label{rem1}
Observe that $\psi_z(x) = \varphi_z(x) + z \partial_z \varphi_z(x)$ for $(x,z)\in (0,\infty)^2$.
\end{remark}

\begin{corollary}\label{cor3}
	For any $f\in\mathcal{C}_1$, 
	\begin{equation}
		\mathcal{J}[f] \ge \frac{1}{\pi} + \inf_{z>0} \int_0^\infty x (f-\varphi_z)^2(x)\ \mathrm{d}x,  \label{CIR10}
	\end{equation}	
	with equality if and only if $f=\varphi_z$ for some $z>0$.
\end{corollary}

\begin{proof}
	By~\eqref{CIM} and Young's inequality, 
	\begin{equation}
		\varphi_z(x)\le \frac{1}{\pi x}, \qquad (x,z)\in (0,\infty)^2, \label{CIR11}
	\end{equation} 
	so that
	\begin{equation*}
		d^2(f,\mathcal{E}) \ge \inf_{z>0} \int_0^\infty x (f-\varphi_z)^2(x)\ \mathrm{d}x
	\end{equation*}and~\eqref{CIR10} readily follows from~\eqref{CIR01}. Next, if $f\in\mathcal{C}_1$ satisfies~\eqref{CIR10} with equality, we infer from~\eqref{CIR06} that
	\begin{equation*}
		\frac{1}{\pi} + \frac{1}{\pi} \int_0^\infty \frac{(f-\varphi_{r_f})^2(x)}{\varphi_{r_f}(x)}\ \mathrm{d}x = \frac{1}{\pi} + \inf_{z>0} \int_0^\infty x (f-\varphi_z)^2(x)\ \mathrm{d}x.
	\end{equation*}	
	Consequently,
	\begin{equation*}
		\frac{1}{\pi} \int_0^\infty \frac{(f-\varphi_{r_f})^2(x)}{\varphi_{r_f}(x)}\ \mathrm{d}x = \inf_{z>0} \int_0^\infty x (f-\varphi_z)^2(x)\ \mathrm{d}x \le \int_0^\infty x (f-\varphi_{r_f})^2(x)\ \mathrm{d}x,
	\end{equation*}
	from which we deduce that
	\begin{equation*}
		\int_0^\infty \left( \frac{1}{\varphi_{r_f}(x)} - \pi x \right) (f-\varphi_{r_f})^2(x)\ \mathrm{d}x \le 0.
	\end{equation*}
	Since the left-hand side of the above inequality is non-negative according to~\eqref{CIR11}, we conclude that necessarily $f=\varphi_{r_f}$, which completes the proof.
\end{proof}

\section{Carlson's discrete inequality and variants}\label{sec3}

We first establish a discrete version of \Cref{prop1}.

\begin{proposition}\label{prop2}
	Let $\theta\in [0,1)$ and $\mathbf{a}\in \mathcal{D}_1$. Then 
	\begin{equation}
		J_\theta[\mathbf{a}] = \left( \sum_{n=1}^\infty a_n^2 \right)^{1/2} \left( \sum_{n=1}^\infty (n-\theta)^2 a_n^2 \right)^{1/2} = \frac{1}{\nu_\theta(r_{\mathbf{a}})} + \frac{1}{\nu_\theta(r_{\mathbf{a}})} \sum_{n=1}^\infty \frac{(a_n-\varphi_{n,r_{\mathbf{a}}})^2}{\varphi_{n,r_{\mathbf{a}}}} \label{DCIR06}
	\end{equation}
	with
	\begin{equation}
		r_{\mathbf{a}}^2 := \left( \sum_{n=1}^\infty a_n^2 \right) \left( \sum_{n=1}^\infty (n-\theta)^2 a_n^2 \right)^{-1} \in \left( 0,\frac{1}{1-\theta} \right] \label{DCIR07}
	\end{equation}
	and
	\begin{equation}
		\varphi_{n,z} := \frac{2}{\nu_\theta(z)} \frac{z}{1+z^2 (n-\theta)^2}, \qquad n\ge 1, \;\; z\in (0,\infty). \label{DCIR08}
	\end{equation}
\end{proposition}

As in the continuous case, we note that~\eqref{DCIR03} and~\eqref{DCIR08} guarantee that $\boldsymbol{\varphi}_z := (\varphi_{n,z})_{n\ge 1}$ belongs to $\mathcal{D}_1$. 

\begin{proof}[Proof of \Cref{prop2}]
	Let $\mathbf{a}\in \mathcal{D}_1$ and introduce
	\begin{equation*}
		q_{2,\mathbf{a}} := \sum_{n=1}^\infty (n-\theta)^2 a_n^2, \quad q_{0,\mathbf{a}} := \sum_{n=1}^\infty a_n^2.
	\end{equation*}
	For $z>0$ and $\boldsymbol{\varphi}=(\varphi_n)_{n\ge 1}\in \mathcal{D}_1$ to be determined later, we perform the same computations as in the proof of \Cref{prop1} and use~\eqref{DCIR08} to obtain the following identity:
	\begin{align*}
		\frac{z q_{2,\mathbf{a}}}{2} + \frac{q_{0,\mathbf{a}}}{2z} & = \sum_{n=1}^\infty \frac{1 + z^2 (n-\theta)^2}{2z} \left[ \varphi_n^2 + 2 \varphi_n (a_n-\varphi_n) +  (a_n-\varphi_n)^2 \right] \\
		& = \sum_{n=1}^\infty \frac{\varphi_n^2 + 2 \varphi_n (a_n-\varphi_n) +  (a_n-\varphi_n)^2}{\nu_\theta(z) \varphi_{n,z}}.
	\end{align*}
	Choosing $\boldsymbol{\varphi} = \boldsymbol{\varphi}_z$, the second term on the right-hand side of the above identity vanishes and we end up with
	\begin{align*}
		\frac{z q_{2,\mathbf{a}}}{2} + \frac{q_{0,\mathbf{a}}}{2z} & = \frac{1}{\nu_\theta(z)} \sum_{n=1}^\infty \varphi_{n,z} + \frac{1}{\nu_\theta(z)} \sum_{n=1}^\infty \frac{(a_n-\varphi_{n,z})^2}{\varphi_{n,z}} \\
		& = \frac{1}{\nu_\theta(z)} + \frac{1}{\nu_\theta(z)} \sum_{n=1}^\infty \frac{(a_n-\varphi_{n,z})^2}{\varphi_{n,z}}.
	\end{align*}
	Choosing $z=r_{\mathbf{a}}$ completes the proof.
\end{proof}

In contrast to the continuous case, an explicit formula for $r_{\boldsymbol{\varphi}_z}$ does not seem to be available. Going further thus requires more precise information on the function $\nu_\theta$, which we derive now. 

\begin{lemma}\label{lem5}
	For $\theta\in [0,1)$ and $z>0$,
	\begin{equation}
		\nu_\theta(z) = i \left[ \psi^{(0)}\left( 1 - \theta - \frac{i}{z} \right) - \psi^{(0)}\left( 1 - \theta + \frac{i}{z} \right) \right] = 2  \Im\left[ \psi^{(0)}\left( 1 - \theta + \frac{i}{z} \right) \right], \label{DCIR09}
	\end{equation}
	where $\psi^{(0)}=\Gamma'/\Gamma$ is the Digamma function. In particular,
	\begin{subequations}\label{DCIR10}
	\begin{equation}
		\nu_\theta(z) = \pi + (2\theta-1) z + o(z) \;\;\text{ as }\;\; z\to 0, \label{DCIR10a}
	\end{equation}
	so that
	\begin{equation}
		\lim_{z\to 0} \nu_\theta(z) = \pi, \quad \lim_{z\to 0} \nu_\theta'(z) = 2\theta-1. \label{DCIR10b}
	\end{equation}
	\end{subequations}
	Moreover,
	\begin{equation}
		\nu_0(z) = - z + \pi \coth\left(\frac{\pi}{z}\right) \;\;\text{ and }\;\; \nu_{1/2}(z) = \pi \tanh\left(\frac{\pi}{z}\right), \qquad z>0. \label{DCIR11}
	\end{equation}
\end{lemma}

\begin{proof}
	Recalling that $\overline{\psi^{(0)}(\xi)}=\psi^{(0)}\big(\bar{\xi}\big)$ for $\xi\in\mathbb{C}$ and 
	\begin{equation*}
		\psi^{(0)}(\xi) = - \gamma + \sum_{n=0}^\infty \frac{\xi-1}{(n+1)(n+\xi)}, \qquad \xi\in\mathbb{C}, \ -\xi\not\in\mathbb{N},
	\end{equation*}
	where $\gamma$ is the Euler-Mascheroni constant, see \cite[Eq.~6.3.9 \& Eq.~6.3.16]{AbSt1964}, a simple computation leads to~\eqref{DCIR09}. Next, denoting the principal value of the logarithm defined on $\mathbb{C}\setminus (-\infty,0]$ by $\ln$, we infer from \cite[Eq.~6.3.18]{AbSt1964} that
	\begin{equation*}
		\psi^{(0)}(\xi) = \ln{\xi} - \frac{1}{2\xi} + o\left( \frac{1}{|\xi|} \right) \;\;\text{ as }\;\; |\xi|\to\infty.
	\end{equation*}
	Since
	\begin{equation*}
		\left| 1 - \theta \pm\frac{i}{z} \right| = \frac{\sqrt{1+z^2(1-\theta)^2}}{z} \sim \frac{1}{z} \;\;\text{ as }\;\; z\to 0, \ z>0, 
	\end{equation*}
	we find
	\begin{align*}
		\nu_\theta(z) & = i \left[ \ln{\left( 1 - \theta - \frac{i}{z} \right)} - \frac{z}{2(z(1-\theta)-i)} - \ln{\left( 1 - \theta + \frac{i}{z} \right)} + \frac{z}{2(z(1-\theta)+i)} \right] +o(z) \\
		& = \frac{i}{2} \ln{\left( \frac{1+z^2(1-\theta)^2}{z^2} \right)} - 2 \arctan{\left( - \frac{1}{z(1-\theta) + \sqrt{1+z^2(1-\theta)^2}} \right)} \\
		& \quad - \frac{iz}{2} \frac{z(1-\theta)+i}{1+z^2(1-\theta)^2} - \frac{i}{2} \ln{\left( \frac{1+z^2(1-\theta)^2}{z^2} \right)} \\
		& \quad + 2 \arctan{\left( \frac{1}{z(1-\theta) + \sqrt{1+z^2(1-\theta)^2}} \right)} + \frac{iz}{2} \frac{z(1-\theta)-i}{1+z^2(1-\theta)^2}+ o(z) \\
		& = 4 \arctan{\left( \sqrt{1+z^2(1-\theta)^2} - z(1-\theta) \right)} + \frac{z}{1+z^2(1-\theta)^2} + o(z) \\
		& = 4 \arctan{(1)} + 4 \arctan'{(1)} \left( \sqrt{1+z^2(1-\theta)^2} - 1 - z(1-\theta) \right) + z + o(z) \\
		& = \pi - 2z(1-\theta) + z + o(z) = \pi + (2\theta-1) z + o(z), 
	\end{align*}
	from which~\eqref{DCIR10} follows. Finally, the explicit formulas~\eqref{DCIR11} are immediate consequences of~\eqref{DCIR09} and \cite[Eq.~6.3.12 \& Eq.~6.3.13]{AbSt1964}. 
\end{proof}

\Cref{lem5} allows us to retrieve additional information on the sequence $\boldsymbol{\varphi}_z$ for $z>0$.

\begin{corollary}\label{cor4}
For $z\in (0,\infty)$, 
\begin{subequations}\label{DCIR12}
\begin{equation}
	\sum_{n=1}^\infty \varphi_{n,z}^2 = \frac{z\big[ \nu_\theta(z) + z\nu_\theta'(z) \big]}{\nu_\theta^2(z)}, \qquad \sum_{n=1}^\infty (n-\theta)^2 \varphi_{n,z}^2 = \frac{\nu_\theta(z) - z\nu_\theta'(z)}{z \nu_\theta^2(z)}, \label{DCIR12a}
\end{equation}
and
\begin{equation} 
	\frac{r_{\boldsymbol{\varphi}_z}^2}{z^2} = \frac{\nu_\theta(z) + z\nu_\theta'(z)}{\nu_\theta(z) - z\nu_\theta'(z)}. \label{DCIR12b}
\end{equation} 
Moreover,
\begin{equation}
	\frac{r_{\boldsymbol{\varphi}_z}^2}{z^2} = 1 + \frac{2}{\pi} (2\theta-1) z + o(z) \;\;\text{ as }\;\; z\to 0. \label{DCIR12c}
\end{equation}
\end{subequations}
\end{corollary}

\begin{proof}
Let $z>0$. Since
\begin{equation}
	\frac{z \nu_\theta'(z) - \nu_\theta(z)}{2z^2} = \frac{\mathrm{d}}{\mathrm{d}z} \left( \frac{\nu_\theta(z)}{2z} \right) = - \sum_{n=1}^\infty \frac{2z (n-\theta)^2}{[1+z^2(n-\theta)^2]^2}, \label{DCIR13}
\end{equation}
we find
\begin{align*}
	\frac{z}{2} \frac{\mathrm{d}}{\mathrm{d}z} \left( \frac{\nu_\theta(z)}{2z} \right) & = - \sum_{n=1}^\infty \frac{(n-\theta)^2}{[1+z^2(n-\theta)^2]^2} = \sum_{n=1}^\infty \frac{1}{[1+z^2(n-\theta)^2]^2} - \sum_{n=1}^\infty \frac{1}{1+z^2(n-\theta)^2} \\
	& = \sum_{n=1}^\infty \frac{1}{[1+z^2(n-\theta)^2]^2} - \frac{\nu_\theta(z)}{2z},
\end{align*}
from which we deduce that
\begin{equation}
	\sum_{n=1}^\infty \frac{1}{[1+z^2(n-\theta)^2]^2}  = \frac{\nu_\theta(z) + z \nu_\theta'(z)}{4z}. \label{DCIR14}
\end{equation}
The formulas stated in~\eqref{DCIR12a} and~\eqref{DCIR12b} then readily follow from~\eqref{DCIR13} and~\eqref{DCIR14}. We then combine~\eqref{DCIR12b} with~\eqref{DCIR10b} to obtain that, as $z\to 0$, 
\begin{equation*}
	\frac{r_{\boldsymbol{\varphi}_z}^2}{z^2} = \frac{\pi + 2(2\theta-1)z +o(z)}{\pi + o(z)} = 1 + \frac{2}{\pi} (2\theta-1) z + o(z),
\end{equation*}
thereby completing the proof.
\end{proof}

Owing to \Cref{prop2} and \Cref{cor4}, we are in a position to identify $\mu(\theta)$ for $\theta\in [0,1)$ and whether it is attained.

\begin{proof}[Proof of \Cref{thm2}]
Let $\theta\in [0,1)$ and consider $\mathbf{a}\in\mathcal{D}_1$. By~\eqref{DCIR07} and \Cref{prop2}, 
\begin{equation*}
	\left( \sum_{n=1}^\infty a_n^2 \right)^{1/2} \left( \sum_{n=1}^\infty (n-\theta)^2 a_n^2 \right)^{1/2} \ge \frac{1}{\nu_\theta(r_{\mathbf{a}})} \ge \frac{1}{\bar{\nu}_\theta},
\end{equation*}
so that
\begin{equation}
	\mu(\theta)\ge \frac{1}{\bar{\nu}_\theta}. \label{DCIR15}
\end{equation}
Next, since $\nu_{\theta_1}<\nu_{\theta_2}$ and $1/(1-\theta_1)<1/(1-\theta_2)$ for $0\le \theta_1<\theta_2<1$, we deduce that
\begin{equation}
	\bar{\nu}_{\theta_1} \le \bar{\nu}_{\theta_2}, \qquad 0 \le \theta_1 < \theta_2 < 1. \label{DCIR16}
\end{equation}
Owing to~\eqref{DCIR11}, $\nu_0$ and $\nu_{1/2}$ are decreasing on $(0,\infty)$, so that $\bar{\nu}_0=\bar{\nu}_{1/2}=\pi$ in view of~\eqref{DCIR10b}. Consequently, by~\eqref{DCIR16},
\begin{subequations}\label{DCIR17}
\begin{equation}
	\bar{\nu}_\theta = \pi, \qquad \theta\in \left[ 0, \frac{1}{2} \right], \label{DCIR17a}
\end{equation}
and
\begin{equation}
	\nu_\theta(z) < \nu_{1/2}(z) < \pi \;\;\text{ for }\;\; z\in \left( 0 , \frac{1}{1-\theta} \right] \;\;\text{ and }\;\; \theta\in \left[ 0, \frac{1}{2} \right]. \label{DCIR17b}
\end{equation}
\end{subequations}

We now split the analysis according to the range of $\theta$.

\smallskip
\noindent$\triangleright$ If $\theta\in [0,1/2]$ then $\bar{\nu}_\theta=\pi$ by~\eqref{DCIR17a}. We pick $z\in (0,\infty)$ and infer from~\eqref{DCIR12a} that
\begin{equation*}
	\mu(\theta) \le J_\theta[\boldsymbol{\varphi}_z] = r_{\boldsymbol{\varphi}_z} \sum_{n=1}^\infty (n-\theta)^2 \varphi_{n,z}^2  = \frac{r_{\boldsymbol{\varphi}_z}}{z} \frac{\nu_\theta(z) - z\nu_\theta'(z)}{\nu_\theta^2(z)},
\end{equation*}
whence
\begin{equation*}
	\mu(\theta) \le \lim_{z\to 0} \frac{r_{\boldsymbol{\varphi}_z}}{z} \frac{\nu_\theta(z) - z\nu_\theta'(z)}{\nu_\theta^2(z)} = \frac{1}{\pi}
\end{equation*}
by~\eqref{DCIR10b} and~\eqref{DCIR12c}. Combining~\eqref{DCIR15} and~\eqref{DCIR17} with the above inequality entails that
\begin{equation}
	\mu(\theta) = \frac{1}{\pi} = \frac{1}{\bar{\nu}_\theta}, \qquad \theta\in \left[ 0,\frac{1}{2} \right]. \label{DCIR18}
\end{equation}

\smallskip
\noindent$\triangleright$ If $\theta\in (1/2,1)$, then $\nu_\theta'>0$ in a right neighbourhood of $z=0$ by~\eqref{DCIR10b}, while
\begin{equation*}
	\nu_\theta'\left(\frac{1}{1-\theta}\right) = 2(1-\theta)^2 \sum_{n=1}^\infty \frac{(1-\theta)^2 - (n-\theta)^2}{\big[(1-\theta)^2 + (n-\theta)^2\big]^2}<0.
\end{equation*}
Therefore, $\bar{\nu}_\theta>\pi$ in view of~\eqref{DCIR10b} and there is $Z\in (0,1/(1-\theta))$ such that $\nu_\theta(Z) = \bar{\nu}_\theta$ and $\nu_\theta'(Z)=0$. This last property implies that
\begin{subequations}\label{DCIR19}
\begin{equation}
	\sum_{n=1}^\infty \frac{1}{\big[1+Z^2(n-\theta)^2\big]^2} = \sum_{n=1}^\infty \frac{Z^2(n-\theta)^2}{\big[1+Z^2(n-\theta)^2\big]^2}, \label{DCIR19a}
\end{equation}
from which we deduce that
\begin{equation}
	r_{\boldsymbol{\varphi}_Z} = Z. \label{DCIR19b}
\end{equation}
\end{subequations}
It then readily follows from \Cref{prop2} (with $\mathbf{a}=\boldsymbol{\varphi}_Z$) and the properties of $Z$ that
\begin{equation*}
	\mu(\theta) \le \left( \sum_{n=1}^\infty \varphi_{n,Z}^2 \right)^{1/2} \left( \sum_{n=1}^\infty (n-\theta)^2 \varphi_{n,Z}^2 \right)^{1/2}  = \frac{1}{\nu_\theta(Z)} = \frac{1}{\bar{\nu}_\theta}.
\end{equation*}
Consequently,
\begin{equation}
	\mu(\theta) = \min_{\mathbf{a}\in\mathcal{D}_1} J_\theta[\mathbf{a}] = \frac{1}{\bar{\nu}_\theta}<\frac{1}{\pi}, \qquad \theta\in \left( \frac{1}{2},1 \right). \label{DCIR20}
\end{equation}

\smallskip

We are left with proving that, for $\theta\in [0,1/2]$, the infimum $\mu(\theta)=1/\pi$ is not attained in $\mathcal{D}_1$. To this end, we assume for contradiction that there is $\mathbf{a}=(a_n)_{n\ge 1}\in \mathcal{D}_1$ such that
\begin{equation}
	J_\theta[\mathbf{a}]^2 = \left( \sum_{n=1}^\infty a_n^2 \right) \left( \sum_{n=1}^\infty n^2 a_n^2 \right) = \frac{1}{\pi^2}. \label{DCIR21}
\end{equation}
Owing to \Cref{prop2}, we infer from~\eqref{DCIR21} that
\begin{equation*}
	\frac{1}{\pi} = \frac{1}{\nu_\theta(r_{\mathbf{a}})} + \frac{1}{\nu_\theta(r_{\mathbf{a}})} \sum_{n=1}^\infty \frac{(a_n-\varphi_{n,r_{\mathbf{a}}})^2}{\varphi_{n,r_{\mathbf{a}}}},
\end{equation*}
with $r_{\mathbf{a}}\in (0,1/(1-\theta)]$ defined in~\eqref{DCIR07}; that is,
\begin{equation*}
	\frac{\pi - \nu_\theta(r_{\mathbf{a}})}{\pi \nu_\theta(r_{\mathbf{a}})} + \frac{1}{\nu_\theta(r_{\mathbf{a}})} \sum_{n=1}^\infty \frac{(a_n-\varphi_{n,r_{\mathbf{a}}})^2}{\varphi_{n,r_{\mathbf{a}}}} = 0.
\end{equation*}
Since both terms on the left-hand side of the above equality are non-negative, we conclude that
\begin{equation*}
	\nu_\theta(r_{\mathbf{a}}) = \pi \;\;\text{ and }\;\; a_n=\varphi_{n,r_{\mathbf{a}}}, \quad n\ge 1,
\end{equation*}
which contradicts the positivity of $r_{\mathbf{a}}$ in view of~\eqref{DCIR17b}. Thus, $\mu(\theta)$ is not attained when $\theta\in [0,1/2]$ and the proof of \Cref{thm2} is complete.
\end{proof}

\begin{remark}\label{rem2}
	For $\theta=0$, simpler proofs of~\Cref{thm2} are available, see \cite{Har1936} for instance, but the approach designed above provides a unified framework to establish \Cref{thm2} for all values of $\theta\in [0,1)$ and turns out to be efficient to study Carlson's inequality for finite sums, as we shall see in the next section.
\end{remark}

\section{Carlson's inequality for finite sums}\label{sec4}

\begin{proof}[Proof of \Cref{thm3}]
Let $\mathbf{a}\in \mathcal{D}_{1,N}$ and set
\begin{equation*}
	 r_{\mathbf{a}}^2 := \left( \sum_{n=1}^N a_n^2 \right) \left( \sum_{n=1}^N n^2 a_n^2 \right)^{-1} \in \left[ \frac{1}{N^2}, 1 \right].
\end{equation*}
We then perform the same computation as in the proof of \Cref{prop2} to derive the identity
\begin{equation}
	S_N[\mathbf{a}] = \left( \sum_{n=1}^N a_n^2 \right)^{1/2}  \left( \sum_{n=1}^N n^2 a_n^2 \right)^{1/2} = \frac{1}{\nu_N(r_{\mathbf{a}})} + \frac{1}{\nu_N(r_{\mathbf{a}})} \sum_{n=1}^N \frac{(a_n-\varphi_{n,r_{\mathbf{a}}})^2}{\varphi_{n,r_{\mathbf{a}}}}. \label{FSCI5}
\end{equation}
Since $r_{\mathbf{a}}\in [1/N,1]$, we readily infer from~\eqref{FSCI5} that
\begin{equation*}
	S_N[\mathbf{a}] \ge \frac{1}{\bar{\nu}_N},
\end{equation*}
whence
\begin{equation}
	\mu_N \ge \frac{1}{\bar{\nu}_N}. \label{FSCI6}
\end{equation}
Next, for $z\in [1/N,1]$,
\begin{equation*}
	\nu_N'(z) = 2 \sum_{n=1}^N \frac{1 - n^2 z^2}{(1 + n^2 z^2)^2},
\end{equation*}
from which we deduce that $\nu_N'(1)<0$ and $\nu_N'(1/N)>0$. Consequently, $\nu_N$ attains its maximum in $(1/N,1)$. Consider now $Z\in (1/N,1)$ such that $\nu_N(Z)=\bar{\nu}_N$. Then $\nu_N'(Z)=0$, which implies that
\begin{align*}
	0 & = \sum_{n=1}^N \frac{1 - n^2 Z^2}{(1 + n^2 Z^2)^2} = \sum_{n=1}^N \big(1 - n^2 Z^2\big) \left( \frac{\bar{\nu}_N \varphi_{n,Z}}{2Z} \right)^2 \\
	& = \frac{\bar{\nu}_N^2}{4 Z^2} \left[ \sum_{n=1}^N \varphi_{n,Z}^2 - Z^2 \sum_{n=1}^N n^2 \varphi_{n,Z}^2 \right] = \frac{\bar{\nu}_N^2}{4 Z^2} \left( r_{\boldsymbol{\varphi}_Z}^2 - Z^2 \right) \sum_{n=1}^N n^2 \varphi_{n,Z}^2. 
\end{align*}
Consequently, $r_{\boldsymbol{\varphi}_Z}=Z$ and it follows from~\eqref{FSCI5} that 
\begin{align*}
	S_N[\boldsymbol{\varphi}_Z] & = \frac{1}{\nu_N(r_{\boldsymbol{\varphi}_Z})} + \frac{1}{\nu_N(r_{\boldsymbol{\varphi}_Z})} \sum_{n=1}^N \frac{(\varphi_{n,Z}-\varphi_{n,r_{\boldsymbol{\varphi}_Z}})^2}{\varphi_{n,r_{\boldsymbol{\varphi}_Z}}} = \frac{1}{\nu_N(Z)} = \frac{1}{\bar{\nu}_N}. 
\end{align*}
Recalling~\eqref{FSCI6}, we have shown that $\mu_N=1/\bar{\nu}_N=S_N[\boldsymbol{\varphi}_Z]$ for all $Z\in (1/N,1)$ such that $\nu_N(Z)=\bar{\nu}_N$. Conversely, if $\mathbf{a}\in\mathcal{D}_{1,N}$ is such that $S_N[\mathbf{a}]=\mu_N$, then~\eqref{FSCI5} implies that
\begin{equation*}
	\mu_N =	\frac{1}{\bar{\nu}_N} = S_N[\mathbf{a}] = \frac{1}{\nu_N(r_{\mathbf{a}})} + \frac{1}{\nu_N(r_{\mathbf{a}})} \sum_{n=1}^N \frac{(a_n-\varphi_{n,r_{\mathbf{a}}})^2}{\varphi_{n,r_{\mathbf{a}}}};
\end{equation*}
that is,
\begin{equation*}
	\frac{\bar{\nu}_N - \nu_N(r_{\mathbf{a}})}{\bar{\nu}_N \nu_N(r_{\mathbf{a}})} + \frac{1}{\nu_N(r_{\mathbf{a}})} \sum_{n=1}^N \frac{(a_n-\varphi_{n,r_{\mathbf{a}}})^2}{\varphi_{n,r_{\mathbf{a}}}} = 0.
\end{equation*}
Since both terms on the left-hand side of the above identity are non-negative, we conclude that
\begin{equation*}
	\nu_N(r_{\mathbf{a}}) = \bar{\nu}_N \;\;\text{ and }\;\; \mathbf{a} = \boldsymbol{\varphi}_{r_{\mathbf{a}}},
\end{equation*}
which entails that there is $Z\in (1/N,1)$ satisfying $\nu_N(Z) = \bar{\nu}_N$ such that $\mathbf{a}=\boldsymbol{\varphi}_Z$. Finally, the lower bound~\eqref{FSCI4} follows from \Cref{prop3} below.
\end{proof}

No explicit formula seems to be available for $\bar{\nu}_N$ for $N\ge 2$ but  more information can be obtained on $\bar{\nu}_N$ (and thus on $\mu_N$). 

\begin{proposition}\label{prop3}
	Recalling the definition~\eqref{Z2} of $Z_2$, there holds $\nu_2(Z_2)=\bar{\nu}_2$ and, for $N\ge 3$, 
	\begin{equation*}
		2 \left[ \arctan{\sqrt{N+1}} - \arctan{\frac{1}{\sqrt{N+1}}} \right] \le \bar{\nu}_N = \frac{1}{\mu_N} \le 2 \arctan(NZ_2).
	\end{equation*}
\end{proposition}

\begin{proof}
For $N=2$,
\begin{equation*}
	\nu_2'(z) = - \frac{20z^6 - z^4 - 5 z^2 - 2}{[(1+z^2)(1+4z^2)]^2}, \qquad z\in \mathbb{R},
\end{equation*}
and has a unique zero $Z_2\in (0,\infty)$ which is given by~\eqref{Z2} and belongs to $(1/2,1)$. Consequently,
\begin{equation}
	\bar{\nu}_2 = \nu_2(Z_2) \;\;\text{ and }\;\; \nu_2'(z) < 0, \qquad z\in (Z_2,1]. \label{FSCI7}
\end{equation}
Next, for $N\ge 3$ and $z\in (1/3,1]$,
\begin{align*}
	\nu_N'(z) & = 2 \sum_{n=1}^N \frac{1 - n^2 z^2}{(1 + n^2 z^2)^2} = \nu_2'(z) + 2 \sum_{n=3}^N \frac{1 - n^2 z^2}{(1 + n^2 z^2)^2} \\
	& \le \nu_2'(z) + 2 \sum_{n=3}^N \frac{1 - 9 z^2}{(1 + n^2 z^2)^2} < \nu_2'(z).
\end{align*}
Since $Z_2>1/2>1/3$, we deduce from~\eqref{FSCI7} and the above inequality that $\nu_N'(z)<\nu_2'(z)\le 0$ for $z\in [Z_2,1]$. Consequently,
\begin{equation}
	\bar{\nu}_N = \max_{(1/N,Z_2)} \nu_N. \label{FSCI8}
\end{equation}
Moreover, for $z\in [1/N,1]$,
\begin{equation*}
	\frac{\arctan{(N+1)z} - \arctan{z}}{z} = \int_1^{N+1} \frac{dx}{1+z^2 x^2} < \frac{\nu_N(z)}{2z} < \int_0^N \frac{dx}{1+z^2 x^2} = \frac{\arctan{Nz}}{z},
\end{equation*}
so that
\begin{equation}
	2 \big[ \arctan{(N+1)z} - \arctan{z} \big] < \nu_N(z) < 2 \arctan{Nz}, \qquad z\in [1/N,1]. \label{FSCI9}
\end{equation}
It then follows from~\eqref{FSCI9} that
\begin{equation*}
	2 \left[ \arctan{\sqrt{N+1}} - \arctan{\frac{1}{\sqrt{N+1}}} \right] = 2 \max_{z\in [1/N,1]} \big\{ \arctan{(N+1)z} - \arctan{z} \big \} \le \bar{\nu}_N,
\end{equation*}
while~\eqref{FSCI8} and~\eqref{FSCI9} imply that
\begin{equation*}
	\bar{\nu}_N \le 2 \arctan(NZ_2),
\end{equation*} 
and the proof is complete.
\end{proof}

\appendix
\section{Auxiliary results}

\begin{lemma}\label{lemap1}
	We define 
	\begin{equation}
		I_n := \int_0^\infty \frac{\mathrm{d}x}{(1+x^2)^n}, \quad n\ge 1, \qquad J_n := \int_0^\infty \frac{x^2}{(1+x^2)^n}\ \mathrm{d}x, \quad n\ge 2. \label{ap01}
	\end{equation}
	Then
	\begin{equation}
		I_1 = \frac{\pi}{2}, \quad I_{n+1} = \frac{2n-1}{2n} I_n, \quad J_{n+1} = \frac{I_n}{2n}, \qquad n\ge 1. \label{ap02}
	\end{equation}
\end{lemma}

\begin{proof}
Let $n\ge 2$. On the one hand, 
\begin{equation*}
	I_n + J_n = I_{n-1}.
\end{equation*}	
On the other hand, 
\begin{equation*}
	J_n = - \int_0^\infty \frac{x}{2(n-1)} \frac{\mathrm{d}}{\mathrm{d}x}\left[ \frac{1}{(1+x^2)^{n-1}}\right]\ \mathrm{d}x = \frac{I_{n-1}}{2(n-1)}.
\end{equation*}
Combining the above two identities gives~\eqref{ap02}.
\end{proof}

\section*{Acknowledgments}

Part of this work was done while enjoying the hospitality of the Innovation Academy for Precision Measurement Science and Technology, Chinese Academy of Sciences, Wuhan and is partially funded by the the Chinese Academy of Sciences President's International Fellowship Initiative Grant No.~2025PVA0101. I thank Dominique Barr\`ere and Jean-Baptiste Hiriart-Urruty for providing useful references.

For the purpose of Open Access, a CC-BY license has been applied by the authors to all versions of this article leading up to the Author Accepted Manuscript.

\section*{Disclosure statement}

The author reports there are no competing interests to declare.

\bibliographystyle{siam}
\bibliography{CarlsonInequality}

\end{document}